\documentclass[reqno,12pt,letterpaper]{amsart}
\usepackage{amsmath,amssymb,amsthm,graphicx,mathrsfs,url}
\usepackage[usenames,dvipsnames]{color}
\usepackage[colorlinks=true,linkcolor=Red,citecolor=Green]{hyperref}

\newcommand{\RR}{{\mathbb R}}

\newtheorem{theo}{Theorem}
\newtheorem{prop}{Proposition}[section]

\newtheorem{lem}[prop]{Lemma}

\numberwithin{equation}{section}

\DeclareMathOperator{\comp}{comp}

\DeclareMathOperator{\Diag}{Diag}
\DeclareMathOperator{\HS}{HS}
\DeclareMathOperator{\Op}{Op}

\DeclareMathOperator{\supp}{supp}
\DeclareMathOperator{\Vol}{Vol}
\DeclareMathOperator{\vol}{vol}
\DeclareMathOperator{\WF}{WF}
\DeclareMathOperator{\Tr}{Tr}
\def\WFh{\WF_h}
\newcommand{\negint}{{\int\negthickspace\negthickspace\negthickspace
\negthinspace -}}

\title[Quantum ergodicity for restrictions]
{Quantum ergodicity for restrictions to hypersurfaces}
\author{Semyon Dyatlov}
\email{dyatlov@math.berkeley.edu}
\author{Maciej Zworski}
\email{zworski@math.berkeley.edu}
\address{Department of Mathematics, Evans Hall, University of California,
Berkeley, CA 94720, USA}

\begin{document}

\begin{abstract}
Quantum ergodicity theorem states that for quantum systems with 
ergodic classical flows, eigenstates are, in average, uniformly 
distributed on energy surfaces. We show that if $ N$  is a hypersurface 
in the position space satisfying a simple dynamical condition,
the restrictions of eigenstates to $ N$ are also quantum ergodic. 
\end{abstract}

\maketitle

%%%%%%%%%%%%%%%%%%%%%%%%%%%%%%%%%%%%%%%%%%%%%%%%%%%%%%%%%%%%%%%%%%%%%%%%%%%%%%%%
%                                 INTRODUCTION                                 %
%%%%%%%%%%%%%%%%%%%%%%%%%%%%%%%%%%%%%%%%%%%%%%%%%%%%%%%%%%%%%%%%%%%%%%%%%%%%%%%%
\section{Introduction}
\label{int}

In a recent paper~\cite{t-z} Toth and Zelditch proved a remarkable
result stating that if $ ( M , g ) $ is a compact manifold with an
ergodic geodesic flow, then quantum ergodicity holds for restrictions
of eigenfunctions to hypersurfaces satisfying a certain dynamical
condition. In an earlier paper \cite{t-zb} they established a similar
result for bounded domains in $ \RR^n $ whose boundaries are piecewise
smooth and whose billiard flows are ergodic.

The purpose of this note is to provide a short proof of a
semiclassical theorem which simultaneously generalizes both results.
Our approach avoids global constructions and calculations by reducing
equidistribution for restrictions to equidistribution in the ambient
manifold. The geometric condition \eqref{eq:h-dynl} enters to obtain a
decorrelation between contributions to the restrictions coming from
different parts of phase space. The proof uses some ideas of
\cite[Appendix~D]{d-g} but we do not refer to any results from that
paper.

For the standard quantum ergodicity result established by  
Shnirelman, Zelditch and Colin de Verdi\`ere, see 
\cite{d-g},\cite{HMR},\cite{t-zb},\cite{t-z} and
references given there.  The case of Riemannian manifolds with
piecewise smooth boundaries was established in a special case by
G\'erard--Leichtnam~\cite{GeLe}, and by Zelditch and the second author
\cite{z-z} in general.  A semiclassical version of quantum ergodicity
was first provided by Helffer--Martinez--Robert~\cite{HMR}.

To make the presentation more clear, in the introduction we will work
in the setting of manifolds without boundary, referring to
Appendix~\ref{s:appendix-a} for modifications in the case of manifolds
with piecewise smooth boundaries.

%
%%%%%%%%%%%%%%%%%%%%%%%%%%%%%% BEGIN FIGURE %%%%%%%%%%%%%%%%%%%%%%%%%%%%%%
%
\begin{figure}
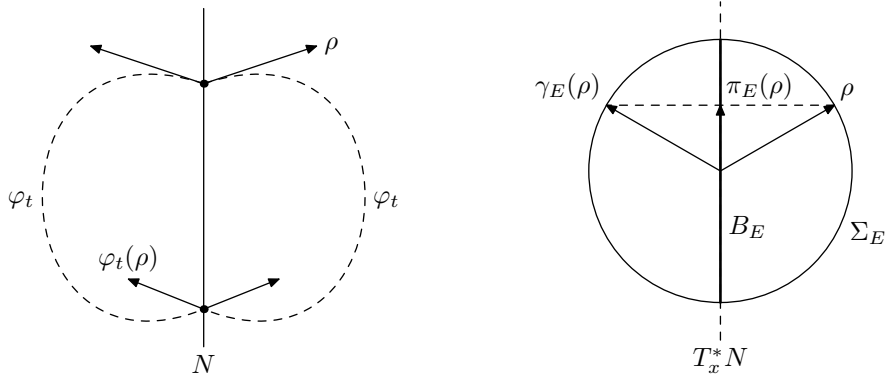

\includegraphics{qeres.1}
\qquad\qquad
\includegraphics{qeres.2}
\caption{Left: the situation prohibited almost everywhere
by the dynamical assumption~\eqref{eq:h-dynl}.
Right: The projection map $\pi_E$ and the reflection map $\gamma_E$
in the cotangent space over some point $x\in N$.}
\label{s:figure-1}
\end{figure}
%
%%%%%%%%%%%%%%%%%%%%%%%%%%%%%%% END FIGURE %%%%%%%%%%%%%%%%%%%%%%%%%%%%%%%
%

Let $ ( M , g ) $ be a compact smooth Riemannian manifold
and consider 
\begin{equation}
\label{eq:Ph}   P ( h ) := -h^2 \Delta_g + V ( x ) , \ \ \ V \in C^\infty ( M;
\mathbb R
) . \end{equation}
The operator  $ P ( h ) $ is self-adjoint when acting 
on half-densities (see \cite[Chapter 9]{e-z}), $ L^2 ( M ,
\Omega_M^{1/2}) $. (This technical point is helpful when more
general operators are considered.)
The classical 
symbol of $ P ( h ) $ is given by 
\[ p ( x , \xi ) = | \xi|_g^2 +  V ( x ) ,  \ \ ( x , \xi ) \in T^* M
, \]
and $ p $ defines the Hamiltonian flow, 
\begin{equation}
\label{eq:varph}   \varphi_t := \exp ( t H_p ) \; : \; p^{-1} ( E ) 
\longrightarrow p^{-1} ( E ) , \ \ E \in \mathbb R, \end{equation}
 We make the following assumption on a range on energies:
\begin{equation}
\label{eq:ener}
\text{For $ E \in [ a, b ] $, $ dp |_{p^{-1} ( E ) } \neq
  0 $, and the flow $ \varphi_t : p^{-1} ( E ) \to p^{-1} ( E
  ) $ is ergodic,}
\end{equation}
where ergodicity is with respect to the Liouville measure 
$\mu_E$ on $ p^{-1} ( E ) $.

Now, let $ N $ be a smooth open hypersurface in $ M $. We
define 
\begin{equation} 
\label{eq:sigla} \Sigma_E :=  p^{-1} ( E ) \cap \pi^{-1}(  N) 
, \end{equation}
where $ \pi : T^* M \to M $ is the natural projection. 
We note that $ \Sigma_E $ is a smooth hypersurface 
in $ p^{-1} ( E ) $ if 
\begin{equation}
\label{eq:VE}
V ( x ) = E \implies  d V ( x ) \notin N^*_x N, 
\end{equation}
and  for simplicity we make this assumption for all $ E \in [ a, b ] $.
For $ E > 0 $ it  is satisfied when $ V \equiv 0 $,  and that is the 
setting of Theorem \ref{t-Lap}.

By restricting
elements of $ \Sigma_E $ to $ TN $ we obtain a map
\begin{equation}
\label{eq:piE}   \pi_E : \Sigma_E \to  B_E := \pi_E ( \Sigma_E ) \subset  T^* N , 
\end{equation}
which is a local diffeomorphism almost everywhere. It defines 
a unique nontrivial involution
\[   \gamma_E :
\Sigma_E \to \Sigma_E,  \ \ 
\pi_E \circ  \gamma_E  = \pi_E , \ \ \gamma_E
\circ \gamma_E = id \]
which 
is the reflection across the orthogonal complement
of the normal bundle $N^*N\subset T_N^*M$ with respect to the metric
on the fibers of $T^*M$ induced by $g$.
This involution enters into the dynamical assumption we make on $ N$
(see Fig.~\ref{s:figure-1}):
\begin{gather}
\label{eq:h-dynl}
\begin{gathered}
\text{For $ E \in [ a, b ] $, the set of $ \rho \in 
\Sigma_E $ satisfying } \varphi_t(\rho)\in\Sigma_E\\
\text{and }\varphi_t ( \gamma_E ( \rho ) ) = \gamma_E ( \varphi_t ( \rho )  ) 
\ \text{ for some $ t \neq 0 $, has measure $ 0 $.}
\end{gathered}
\end{gather}

We denote by $ u_j ( h ) $ a normalized eigenfunction of $ P  ( h ) $
with an eigenvalue $ E_j ( h ) $,  
\[ (P ( h ) - E_j ( h ))u_j(h) = 0 , \ \  \| u_j ( h ) \|_{L^2 ( M ,
  \Omega_M^{\frac12}) }  = 1 .\]

To formulate the quantum ergodicity theorem for restrictions, we need
to restrict half-densities to $ N $ and that requires making a
choice. Suppose $ f \in C^\infty ( M ) $, $ f|_N = 0 $, $ df|_N \neq 0
$.  Informally, the restriction is now defined using,
$  
| d x|^{\frac12} =  | dy |^{\frac 12} |d f|^{\frac12} $ , $ x \in M $, 
$  y \in N $.
More precisely if, in local coordinates, $ x = (  x', x_n ) $, 
$ N = \{ x_n = 0 \} $ then, in the half-density notation 
of \cite[\S 9.1]{e-z},
\begin{equation}
\label{eq:resth}  \big( u ( x ) | dx|^{\frac12} \big) |_N :=  u ( x', 0 )
|dx'|^{\frac12}  \left| \frac {\partial f}{ \partial x_n}  ( x', 0)
\right|^{-\frac 12} . 
\end{equation}

Using the notation of~\cite[Chapter 14]{e-z}, reviewed in
Section~\ref{pr} below, we can state our main result. {See
Appendix~\ref{s:appendix-a} for the modifications needed in the case
when $M$ has a boundary.}
\begin{theo}
\label{t-Sch}
Suppose that $( M , g ) $ is a compact Riemannian manifold with a
piecewise smooth boundary satisfying~\eqref{eq:VE2} and that $ u_j =
u_j ( h) $ are normalized eigenfuctions of the Dirichlet realization
of~$ P (h)$.

Suppose also that \eqref{eq:ener} holds and that $ N $ is a smooth
open hypersurface not intersecting $\partial M$ and
satisfying~\eqref{eq:h-dynl}. For $ Q \in \Psi^m_h ( M , \Omega_M^{1/2}
) $ put $ v_j := Q u_j ( h ) |_ N$, where the restriction operator on
half densities is defined in~\eqref{eq:resth}.  Then for $ A \in
\Psi^0_h( N , \Omega_N^{1/2} ) $, compactly supported in $ N $, we have
\begin{equation}
  \label{eq:qe1}
h^n \sum_{ E_j  \in [ a, b ] } 
\bigg|  \langle A v_j  , v_j  \rangle_{ L^2 ( N , \Omega^{1/2}_N )
  }  - {\int \!\! \! }_{ \Sigma_{E_j}  }
\pi_{E_j  } ^ *\sigma ( A )
 | \sigma ( Q ) |^2  \,d \nu_{E_j  } 
 \bigg| \longrightarrow 0 ,  \
\ \ h \to 0 ,  \end{equation}
where $ \sigma ( A ) \in S^0 ( T^*N ) $ is the symbol of $ A $, $
\sigma ( Q ) \in S^m ( T^* M ) $ is the symbol of $ Q $, and the
measure $\nu_E$ is defined on $\Sigma_E$ by the identity
\begin{equation}
\label{eq:mL}
{d\mu_E\over\mu_E(p^{-1}(E))}=d\nu_E\wedge df,
\end{equation}
with $ \mu_E $ the Liouville measure and $ f $ defining the restriction of half-densities in \eqref{eq:resth}.
\end{theo}

\medskip\noindent
{\bf Remarks.}
(i) The measure $\nu_E$ has a particularly nice description in the
case $V=0$. Assume that $ E = 1 $, then $ \Sigma_1 = S_N^* M $ where $
S^*_N M \subset S^*M $ denotes the cosphere bundle of $M$ restricted
to $N$. The Liouville measure, $ \mu_1 $, on $ S^* M $, induces, for
each $ x \in M $, a measure on $ S_x^* M $, $ \mu_x $, such that
\[  \mu_1 ( \Omega ) = \int_M  \mu_x ( \Omega \cap S_x^* M ) \,d \vol_g (x) , \
\ \Omega \subset S^*M . \]
Our measure on $ \Sigma_1 = S^*_N M $ is then given by 
\begin{equation}
\label{eq:mN} 
\nu_1  ( \Gamma ) =  \frac{ 1 } { \mu_1 ( S^* M ) } \int_N \mu_x ( \Gamma \cap S^*_x M ) \,d \vol_{g |_N} ( x
) ,  \ \ \Gamma \subset S_N^* M , 
\end{equation}
where $ g|_N $ is the metric on $ N $ induced by $ g $. (Here $
\Omega $ and $ \Gamma $ are Borel sets.) See Appendix~\ref{s:appendix-b}
for details.

\medskip\noindent
(ii) The now standard argument due to Colin de Verdi\`ere and Zelditch
and described in~\cite[Theorem~15.5]{e-z} shows that this result
provides pointwise convergence for a density one subsequence: there
exists a family of sets, $\Lambda(h) \subset \{a \le E_j \le b\}$,
such that $\lim_{h\to 0 } {\# \Lambda(h)\over \#\{a \le E_j \le b\}} =
1$, and, in the notation of Theorem \ref{t-Sch},
\begin{equation}
\label{eq8.18}
\sup_{E_j\in\Lambda(h)}\bigg|\langle A v_j, v_j\rangle_{L^2 ( N, \Omega_N^{1/2} ) } 
- {\int \!\! \! }_{ \Sigma_{E_j}  }
\pi_{E_j  } ^ *\sigma ( A )
 | \sigma ( Q ) |^2  \,d \nu_{E_j  }\bigg|\longrightarrow 0 \,, \ \
 \mbox{as }  h \to 0 .
\end{equation}
In step~4 of the proof of~\cite[Theorem~15.5]{e-z}, to pass
from~\eqref{eq8.18} for a countable dense family $A_k$ of
pseudodifferential operators to the full statement, one needs uniform
boundedness of $v_j$ on $L^2(N)$.  While this need not be true for the
whole sequence, it holds for $E_j\in\Lambda(h)$ if we take one of the
operators $A_k$ to be the identity.

\medskip
\noindent
(iii) The dynamical condition of Toth--Zelditch~\cite{t-z} is stated
using Poincar\'e return times but the analysis in that paper shows
that it is equivalent to our condition~\eqref{eq:h-dynl}. The
paper~\cite{t-z} provides interesting examples for which it is
satisfied.

\medskip
\noindent
(iv) We sum over eigenvalues in a fixed size interval $[a,b]$ in~\eqref{eq:qe1}
since the corresponding smoothed out spectral projectors are pseudodifferential operators.
It would be interesting to prove an analogous statement for size $h$ intervals,
whose projectors have more complicated microlocal structure,
as in~\cite[Appendix~D]{d-g}.

\medskip

The paper is organized as follows. In Section~\ref{pr} we review basic
concepts of semiclassical quantization and present slightly
non-standard facts needed in the proof. The key point is that even in
the case of manifolds with boundary we only need to work with standard
pseudodifferential operators. In Section~\ref{dec} we present a
general decorrelation result and in Section~\ref{qer} the proof of
Theorem \ref{t-Sch}. Since no reference for quantum ergodicity for
semiclassical boundary value problems seems to be available,
Appendix~\ref{s:appendix-a} present a proof in the spirit
of~\cite{z-z} with some simplifications based
on~\cite[Appendix~D]{d-g} (see Lemma~\ref{l:1}). The high energy
result for Laplacians, as presented in~\cite{t-z}, follows from
Theorem~\ref{t-Sch} but there is an issue at energy $ 0 $, and
Appendix~\ref{s:appendix-b} shows how that is overcome.

Except for the 
efficiency provided by direct semiclassical methods \cite{e-z}, the
proofs are similar to those in~\cite{t-z}. The one significant
difference is the treatment of quantum ergodicity for microlocal
Cauchy data in Section~\ref{qer}~-- see~\cite{c-t-z} for
comparison. 
%For future groundbreaking contributions
%to this exciting subject the reader is advised
%to look forward to the appearance of~\cite{c-t-z2}.

%%%%%%%%%%%%%%%%%%%%%%%%%%%%%%%%%%%%%%%%%%%%%%%%%%%%%%%%%%%%%%%%%%%%%%%%%%%%%%%%
\smallskip\noindent{\sc Acknowledgements}.
We would like to thank Nicolas Burq, Oran Gannot and St\'ephane
Nonnenmacher for helpful comments on the first version of this paper,
Hart Smith for explaining an alternative proof of Lemma \ref{l:fPb}
based on finite speed of propagation, Steve Zelditch for
encouraging us to handle the boundary case, and an anonymous referee
for many helpful suggestions. The partial support by
National Science Foundation under the grant DMS-1201417 is also
gratefully acknowledged.

%%%%%%%%%%%%%%%%%%%%%%%%%%%%%%%%%%%%%%%%%%%%%%%%%%%%%%%%%%%%%%%%%%%%%%%%%%%%%%%%
%%%%%%%%%%%%%%%%%%%%%%%%%%%%%%%%%%%%%%%%%%%%%%%%%%%%%%%%%%%%%%%%%%%%%%%%%%%%%%%%
\section{Semiclassical preliminaries}
\label{pr}

We will use the calculus of semiclassical pseudodifferential operators
described in~\cite[\S 9.3, \S 14.2]{e-z}. Our operators will always be
supported away from the boundary of the manifold and hence can be
considered as operators on a boundaryless manifold {(for example, by
considering the double space of a neighborhood of our manifold in a
slightly bigger open manifold without boundary). A notable exception
is the Schr\"odinger operator $-h^2\Delta_g+V(x)$, which can however be
extended smoothly past the boundary.}

For a compact manifold, $ X$ (which could be different from the
compact manifold $ M$ considered in Section~\ref{int}), the class $
\Psi^m_h (X ) $ denotes operators of order $ m$, so that, for instance $
- h^2 \Delta_g \in \Psi^2_h ( M ) $. We have the symbol map, $ \sigma $,
appearing in the following exact sequence
\[
0 \longrightarrow h \Psi^{m-1}_h ( X )  \longrightarrow \Psi^{m}_h(X)
{\stackrel{ \sigma } {\longrightarrow } } \; S^m ( T^* X ) / h S^{m-1} ( T^*
X ) \longrightarrow 0  ,
\]
where $ S^m $ denotes the standard space of symbols. We take some
quantization map $ \Op_h: S^m ( T^* X ) \to \Psi^m_h $; it satisfies 
\[
\sigma( \Op_h ( a ) ) = a \mod h S^{m-1} ( T^* X ).
\]

We also introduce the class of {\em compactly microlocalized
pseudodifferential operators}, $ \Psi^{\comp}_h ( X ) $: $ A \in
\Psi^{-\infty } ( X ) $ is in $ \Psi^{\comp}_h ( X ) $ if for some $
\chi \in C^\infty_{\mathrm{c}} ( T^* X ) $,
\[
\Op_h (1 -  \chi ) A \in h^\infty \Psi^{-\infty } ( X ) .
\]

For this class the definition of $ \WFh ( A ) $ given in~\cite[\S
8.4]{e-z} applies. From the same section we take the definition of
microlocal equality of operators.

Following \cite[\S 2.3]{d}, \cite[\S 3]{s-z}, and~\cite[\S 11.2]{e-z}
we consider Fourier integral operators quantizing a canonical
transformation $ \kappa : U_1 \to U_2 $, $ U_1 \Subset T^*X $ and $
U_2 \Subset T^*Y $, $ \kappa $ defined on a neighbourhood of $ U_1 $:
we say that an operator $ F : L^2 ( X ) \to L^2 ( Y ) $, quantizes $
\kappa $ if for any $ A \in \Psi^{\comp}_h ( Y ) $ with $ \WFh ( A )
\Subset U_2 $,
\begin{equation}
\label{eq:Eg}     F^* A F = B , \ \ B \in \Psi^{\comp}_h ( X ) , \ \ 
\sigma ( B ) = \kappa^* \sigma ( A ) .
\end{equation}
We further require that $F$ be microlocally unitary in the sense that
$ F^{-1} = F^* $ microlocally near $ U_1 \times U_2 $. If $F$
quantizes $\kappa$, then the operator $ F^* $ quantizes $ \kappa^{-1}
$. More generally we say that $ G $ is a Fourier integral operator
associated to the canonical relation $\kappa $ if $ G = A F $, for $ F
$ above and some $ A \in \Psi^0_h ( X ) $.

The standard example is given by $ F( t ) = e^{- it P ( h )/h } $,
where $ P ( h ) = -h^2 \Delta_g + V ( x ) \in \Psi^2_h ( M ) $ (or a
more general operator) which quantizes the Hamiltonian flow $
\varphi_t := \exp(tH_p)$.

We say that a tempered operator (see \cite[\S 8.4]{e-z}) $ G: L^2 ( X
) \to L^2 (Y ) $, is {\em compactly microlocalized} if for some $ A
\in \Psi_h^{\comp } ( Y ) $ and $ B \in \Psi_h^{\comp } ( X ) $,
\begin{equation}
\label{eq:comi}    A G B - G \in h^\infty \Psi^{-\infty } . 
\end{equation}
In that case we can define $ \WFh' ( G ) \subset T^*Y \times T^*X$, by
taking the twisted $ \WFh $ of its Schwartz kernel, $ K_G $:
\begin{equation}
  \label{e:wf-h-operator}
 \WFh' ( G ) := \{ ( y , \eta; x , -\xi ) \; : \;   ( y , \eta ; x , \xi  )
\in \WFh ( K_G ) \}.
\end{equation}
If $F$ is associated to some canonical transformation $\kappa$, then
$\WFh'(F)$ lies inside the graph of~$\kappa$.

We recall from \cite[Theorem 14.9]{e-z} that if $ M $ has no boundary,
and $ P ( h ) = -h^2 \Delta + V ( x ) $, $ f \in C^\infty_{\mathrm c}
( \mathbb R ) $, then
\begin{equation}
\label{eq:fP}   f ( P ( h ) ) \in \Psi_h^{\comp } ( M ) , \ \ \sigma \big( f ( P ( h ) ) \big)
= f ( p ) .
\end{equation}

We now present three lemmas which will be used in the paper.  We
assume that the manifold $ M $ and operator $ P ( h ) $ are as in
Appendix~\ref{s:appendix-a}.

The first lemma makes an observation that $ f ( P ) $ is a nice
operator away from the boundary.
\begin{lem}
\label{l:fPb}
Suppose that $ f \in C^\infty_{\rm{c}} ( \RR ) $, $ \chi \in
C^\infty_{\rm{c}}  ( M^\circ ) $. Then 
\begin{equation}
\label{eq:fPb}   f ( P ( h ) ) \chi = A_f + {\mathcal
  O}_{L^2 ( M ) \to L^2 ( M ) } ( h^\infty ), \ \ 
\end{equation}
where $ A_f \in \Psi_h^{\comp } ( M^\circ ) $ {is compactly supported
away from the boundary} and its principal 
symbol is given by $ f ( p ( x , \xi ) ) \chi ( x ) $.
\end{lem}
\begin{proof}
Let $ \psi_0 \in C^\infty_{\rm{c}} ( ( - \delta, \delta )  ) $ be equal to $ 1 $ 
near $ 0 $. We write $ f ( P ) \chi $ using the Schr\"odinger propagator, 
\[
f ( P ) \chi = \frac{ 1 } { 2 \pi h } \int_\RR \hat f ( t/h ) e^{ i
  t P / h } \chi  \,dt= 
\frac{ 1 } { 2 \pi h } \int_\RR \hat f ( t/h ) \psi_0 ( t )  e^{ i
  t P / h } \chi\,  dt +  {\mathcal
  O}_{L^2 ( M ) \to L^2 ( M ) } ( h^\infty ) ,
\]
where the error estimate follows from the decay of $ \hat f $ and the
untarity of $ e^{ it P /h }$.  We then write
\[
f ( P ) \chi = \frac{ 1 } { 2 \pi h } \int_\RR \int_\RR f (
\tau )  e^{ \frac {it} h ( 
 P  - \tau )  } \chi  \psi_0 ( t ) \, dt d \tau.
\]
Choose $ \delta $ sufficiently small so that for $|t|\leq\delta$, $p (
x, \xi ) \in \supp f $, and $ x \in \supp \chi $, we have $ d_g ( \pi
( \varphi_t ( x, \xi)) , \partial M ) >\epsilon > 0 $.  In that case
we can use the local parametrix for $ e^{ i t P / h } $, see for
instance \cite[\S 10.2]{e-z}. An application of the stationary phase
method in $ ( t, \tau ) $ variables gives the conclusion of the lemma.
\end{proof}

As before, let $(u_j(h))_{j\in \mathbb N}$ be the full orthonormal
system of eigenfunctions of $P(h)$ with eigenvalues $E_j(h)$. Lemma
\ref{l:fPb} applied to the operator $ f ( P ( h ) ) A $, where $ A \in
\Psi_h^m (M)$ is supported away from the boundary, and $f\in C_{\mathrm
c}^\infty(\mathbb R)$, gives, together with the trace
formula~\cite[Theorem~14.10]{e-z}
\begin{equation}
  \label{e:trace-xp}
(2\pi h)^n\sum_j f(E_j)\langle Au_j,u_j\rangle
=\int_{T^*M} f(p)\sigma(A)\, d \mu_\sigma  +\mathcal O(h), 
\end{equation}
where $ \mu_\sigma $ is the symplectic measure, $\mu_\sigma = {
\sigma^n } /{ n!}  $.

The second lemma, in the spirit of \cite[Appendix~D]{d-g}, gives
estimates using $ L^2 $ norms of symbols:

%
%%%%%%%%%%%%%%%%%%%%%%%%%%%%%% BEGIN PROP %%%%%%%%%%%%%%%%%%%%%%%%%%%%%%
%
\begin{lem}
\label{l:1}
Suppose that $ A\in\Psi_h^m(M^\circ) $ is a pseudodifferential operator compactly supported away 
from $ \partial M $. 
Then for each $a'<a<b<b'$,
\begin{equation}
  \label{e:hs-estimate}
(2\pi h)^n\sum_{E_j\in [a,b]}\|Au_j\|_{L^2}^2\leq \|\sigma(A)\|^2_{L^2(p^{-1}([a',b']))}+\mathcal O(h)
\end{equation}
where the $L^2$ norm of $ \sigma ( A ) $ is taken with respect to the
measure $ \mu_\sigma $.

More generally, if $N\subset M$ is a fixed smooth submanifold (of any
dimension) not intersecting $\partial M$, then there exists a constant
$C$ such that for each $\widetilde A\in\Psi_h^m(N)$ supported in a fixed
compact subset of $N$,
\begin{equation}
  \label{e:hs-estimate-2}
h^n\sum_{E_j\in [a,b]}\|\widetilde A(u_j|_N)\|_{L^2}^2\leq 
C\|\sigma(\widetilde A)\|^2_{L^2(\pi(p^{-1}([a',b'])\cap T^*_NM))}+\mathcal O(h).
\end{equation}
Here $T^*_NM$ is the cotangent bundle of $M$ restricted to $N$ and
$\pi:T^*_NM\to T^*N$ is the projection. 
\end{lem}

\noindent
{\bf Remark.} We note that in the case when  $ \widetilde A =
1 $ we recover the bound
\begin{equation}
\label{e:hs-estimate-3} h^n\sum_{E_j\in [a,b]}\|u_j|_N \|_{L^2}^2\leq 
C \,.\end{equation}
By contrast, for individual eigenfuctions
the bound $ C h^{ \frac{n-k}2 } $ is optimal -- see \cite{b-g-t}
and \cite{ta}.
%When $ N=\{x\} $ is a point we obtain a pointwise bound
%with the constant independent on $x$ (as follows from the proof): 
%$$
%h^n\sup_{x\in M}\sum_{E_j\in [a,b]} | u ( x) |^2 \leq  C.
%$$
%%%%%%%%%%%%%%%%%%%%%%%%%%%%%%%%%%%%%%%%%%%%%%%%%%%%%%%%%%%%%%%%%%%%%%%%%%%%%%%%
\begin{proof}
To show~\eqref{e:hs-estimate}, take $f\in C_{\mathrm c}^\infty(a',b')$
such that $0\leq f\leq 1$ everywhere and $f=1$ on $[a,b]$. Then we write
by~\eqref{e:trace-xp},
$$
\begin{gathered}
(2\pi h)^n\sum_{E_j\in [a,b]}\|Au_j\|_{L^2}^2
\leq (2\pi h)^n\sum_j f(E_j)\langle A^*Au_j,u_j\rangle\\
=\int_{T^*M}f(p)|\sigma(A)|^2\,d \mu_\sigma+\mathcal O(h)
\leq \int_{p^{-1}([a',b'])} |\sigma(A)|^2\, d\mu_\sigma +\mathcal O(h).
\end{gathered}
$$
To show~\eqref{e:hs-estimate-2}, denote by $R_N:C^\infty(M)\to C^\infty(N)$
the restriction operator and note that
\[ \begin{split} 
h^n\sum_{E_j\in [a,b]}\|\widetilde A(u_j|_N)\|_{L^2}^2
& \leq h^n\sum_j |f(E_j)|^2\|\widetilde AR_N u_j\|_{L^2}^2 
 = h^n \sum_j \| \widetilde AR_N f ( P ) u_j\|_{L^2}^2 \\
& =h^n\|\widetilde A R_N f(P)\|_{\HS}^2.
\end{split} \]
The Hilbert--Schmidt norm on the right-hand side is equal to the $L^2$
norm of the Schwartz kernel $K$ of $\widetilde A R_N f(P)$. Recall
that $f(P)$ is a pseudodifferential operator when localized away from
the boundary, by Lemma~\ref{l:fPb}.  Note that $K=\mathcal O
(h^\infty) $ away from the diagonal of $N$ embedded in $N\times M$.
To estimate $K$ near the diagonal, we choose local coordinates
$x=(x',x'')$, $x'\in \mathbb R^k$, $x''\in
\mathbb R^{n-k}$, where $k=\dim N$, on $M$ near some point of $N$, in
which $N$ is given by $\{x''=0\}$.  If $\tilde a$ is the
full symbol of $\widetilde A$ in these coordinates (in the standard
quantization) and $\tilde b$ is the full symbol of the
pseudodifferential operator $f(P(h))$, then we can write
$$
K(x,y)=(2\pi h)^{-n-k}\int e^{{\frac i  h}(x\cdot\eta-z\cdot\eta+z\cdot\xi'-y\cdot\xi)}
\tilde a(x,\eta)\tilde b(z,0,\xi)\,dzd\eta d\xi,
$$
here $y,\xi\in \mathbb R^n$ and $x,z,\eta\in \mathbb R^k$. By the
unitarity of the (semiclassical) Fourier transform, the $L^2_{x,y}$
norm of $K(x,y)$ is equal to the $L^2_{x,\xi}$ norm of
$$
K_1(x,\xi)=(2\pi h)^{-n/2-k}\int e^{{\frac i h}(x\cdot\eta-z\cdot\eta+z\cdot\xi')}
\tilde a(x,\eta)\tilde b(z,0,\xi)\,dz d\eta.
$$
The method of stationary phase shows that
$$
K_1(x,\xi)=(2\pi h)^{-n/2}e^{{\frac i  h}x\cdot \xi'}(\tilde a(x,\xi')\tilde b(x,0,\xi)+\mathcal  O _{C^\infty} (h)).
$$
Now, $h^{n/2}$ times the $L^2$ norm of $K_1$ is bounded by a constant
times the $L^2$ norm of $\tilde a$ on the set $\pi(\supp \tilde b\cap
T^*_NM)$, with an $\mathcal O(h)$ remainder.
\end{proof}
%
%%%%%%%%%%%%%%%%%%%%%%%%%%%%%%% END PROP %%%%%%%%%%%%%%%%%%%%%%%%%%%%%%%
%
To formulate the next lemma we define 
\[ \Diag(T^*M) :=  \{ ( \rho, \rho ) : \rho \in T^* M \} \subset T^*
M \times T^*M . \]

\begin{lem}
\label{l:2}
Suppose that $G : L^2 ( M ) \to L^2 ( M ) $ is a compactly
microlocalized tempered operator in the sense of \eqref{eq:comi},
compactly supported away from the boundary, and that $f\in C_{\mathrm
c}^\infty(\mathbb R)$.  Then for $ G $ satisfying
\[
\WFh'(G)\cap\Diag(T^*M)=\emptyset ,
\]
we have
\begin{equation}
  \label{e:no-trace}
\sum_j f(E_j)\langle G u_j,u_j\rangle=\mathcal O(h^\infty).
\end{equation}
\end{lem}
\begin{proof}
The left-hand side of~\eqref{e:no-trace} is equal to the trace of
$Gf(P(h))$.  We can write $G$ as a finite sum of operators of the form
$X_1GX_2$, where $X_1,X_2\in\Psi_h^{\comp}$ satisfy
$\WFh(X_1)\cap\WFh(X_2)=\emptyset$. Then by the cyclicity of the
trace,
$$
\Tr(X_1GX_2f(P))=\Tr(X_2f(P)X_1G)=\mathcal O(h^\infty),
$$
as $X_2f(P)X_1 \in h^\infty\Psi^{-\infty}$, by Lemma~\ref{l:fPb}.
\end{proof}

%%%%%%%%%%%%%%%%%%%%%%%%%%%%%%%%%%%%%%%%%%%%%%%%%%%%%%%%%%%%%%%%%%%%%%%%%%%%%%%%
%%%%%%%%%%%%%%%%%%%%%%%%%%%%%%%%%%%%%%%%%%%%%%%%%%%%%%%%%%%%%%%%%%%%%%%%%%%%%%%%
\section{Decorrelation for Fourier integral operators}
\label{dec}

In the proof of Theorem~\ref{t-Sch} we will encounter expressions
involving $ \langle F u_j , u_j\rangle$, where $(u_j(h))_{j\in \mathbb
N}$ is the full orthonormal system of eigenfunctions of
$P(h)=-h^2\Delta_g+V(x)$ with eigenvalues $E_j(h)$, and $ F $ is a
compactly microlocalized semiclassical Fourier integral operator. This
section shows that the sum of such terms over $j$ in an $\mathcal
O(1)$ sized spectral window is negligible when the canonical relation
of $ F $ satisfies a `nonreturning' assumption; we call this
phenomenon \emph{decorrelation} for Fourier integral operators.

Assume $F$ is a compactly microlocalized tempered operator $L^2(M)\to L^2(M)$,
in the sense of~\eqref{eq:comi}, and
\begin{equation}
  \label{e:f-assumption}
\|F\|_{L^2\to L^2}=\mathcal O(1),\
\WFh'(F)\subset \{(\kappa(\rho),\rho): \rho\in K_1\},
\end{equation}
where the wavefront set is defined in~\eqref{e:wf-h-operator},
$\kappa:V_1\to V_2$ is a canonical transformation, $V_1,V_2\subset
T^*M$ are open sets, and $K_j\subset V_j$ are compact sets such that
$\kappa(K_1)=K_2$. (In our case, $F$ will be a Fourier integral
operator, but this is not required in the proof.)

For each $t\in \mathbb R$, define the $t$-exceptional set, 
\begin{equation}
\label{eq:exc}
\mathcal E_\kappa(t):= \{ \rho \in K_1\cap \varphi_{-t} (K_1)  :   \varphi_t(\kappa(\rho))
=\kappa(\varphi_t ( \rho ))\}, \ \ \varphi_t := \exp ( t H_p )  .
\end{equation}
Here $\varphi_t$ is the broken Hamiltonian flow of $p$, defined almost everywhere
and described in Appendix~\ref{s:appendix-a}.

The decorrelation result is given as follows:
%
%%%%%%%%%%%%%%%%%%%%%%%%%%%%%% BEGIN PROP %%%%%%%%%%%%%%%%%%%%%%%%%%%%%%
%
\begin{lem}\label{l:fio-decorrelation}
Suppose that $ a < b $ are fixed, and that 
there exists $t_0>0$ such that
\begin{equation}
\label{eq:nonret}
\mu_\sigma\Big(p^{-1}([a,b])\cap \bigcup_{|t|\geq t_0}\mathcal
E_\kappa(t)\Big)=0, 
\end{equation}
where $ {\mathcal E}_\kappa ( t ) $ is
given by \eqref{eq:exc} and~$\mu_\sigma$ is the symplectic measure.

Then for each $F$ compactly supported inside $M^\circ$ and satisfying~\eqref{e:f-assumption},
\begin{equation}
\label{eq:prop}
h^n\sum_{E_j\in [a,b]}|\langle Fu_j,u_j\rangle|\to 0\text{ as }h\to 0.
\end{equation}
\end{lem}

%%%%%%%%%%%%%%%%%%%%%%%%%%%%%%%%%%%%%%%%%%%%%%%%%%%%%%%%%%%%%%%%%%%%%%%%%%%%%%%%
\begin{proof}
Take $T>t_0$ and denote
\begin{equation}
\label{eq:defKT}
\widetilde K_T:=\big((\mathcal B_T\cup \kappa^{-1}(\mathcal B_T))\cap K_1\big)
\cup \bigcup_{t_0\leq |t|\leq T}\mathcal E_\kappa(t).
\end{equation}
Here $\mathcal B_T$ is the set of points at which the broken
Hamiltonian flow is not well-defined on the interval $[-T,T]$,
see~\eqref{e:B-T}.  Then $\widetilde K_T$ is a compact subset of $U_1$
and $\mu_\sigma(\widetilde K_T\cap p^{-1}([a,b]))=0$.  Therefore,
there exists an open set $\widetilde U_T\subset U_1$ and constants
$a'<a$ and $b'>b$ such that
$$
\widetilde K_T\subset\widetilde U_T,\ \ \ \ \ 
\mu_\sigma(\widetilde U_T\cap p^{-1}([a',b']))\leq T^{-1}.
$$
Take $X_T\in\Psi_h^{\comp}(M)$ satisfying $|\sigma(X_T)|\leq 1$,
$\WFh(X_T)\subset \widetilde U_T$, and $X_T=1$ microlocally near
$\widetilde K_T$. Since $ F $ is bounded on $ L^2 ( M ) $, $ |\langle
F X_T u_j,u_j\rangle| \leq C\| X_T u_j \|_{L^2} $.
Hence~\eqref{e:hs-estimate} and~\eqref{e:weyl-law} give
\begin{equation}
  \label{e:decor-1}
\begin{gathered}
h^n\sum_{E_j\in [a,b]}|\langle F X_T u_j,u_j\rangle|
\leq C\bigg(h^n\sum_{E_j\in [a,b]}\| X_T u_j\|_{L^2}^2\bigg)^{\frac12}\\
\leq C(\|\sigma(X_T)\|_{L^2(p^{-1}([a',b']))}^2+\mathcal O_T(h))^{\frac12}
\leq C(T^{-1}+\mathcal O_T(h))^{\frac12},
\end{gathered}
\end{equation}
where $C$ denotes a constant independent of $T$ and $h$.

We now analyse the contribution of
$
F_1:=F(1-X_T).
$
For that define
$$
\langle F_1\rangle_T:=\frac1 T\int_0^T e^{itP(h)/h}F_1e^{-itP(h)/h}\,dt.
$$
For each eigenfunction $u_j$, we have
$$
\langle F_1 u_j,u_j\rangle=\langle\langle F_1\rangle_T u_j,u_j\rangle.
$$
We now take some $f\in C_{\mathrm c}^\infty(\mathbb R)$ such that $0\leq f\leq 1$
everywhere and $f=1$ near $[a,b]$. Then
by~\eqref{e:weyl-law},
\begin{equation}
  \label{e:decor-2}
\begin{gathered}
h^n\sum_{E_j\in[a,b]}|\langle F_1u_j,u_j\rangle|
=h^n\sum_{E_j\in [a,b]}|\langle \langle F_1\rangle_T u_j,u_j\rangle|
\\\leq C\bigg(h^n\sum_j f(E_j)\|\langle F_1\rangle_T u_j\|_{L^2}^2\bigg)^{\frac12}.
\end{gathered}
\end{equation}
We now write
$$
\begin{gathered}
h^n\sum_jf(E_j)\|\langle F_1\rangle_T u_j\|_{L^2}^2
=h^n\sum_jf(E_j)\langle\langle F_1\rangle_T^*\langle F_1\rangle_T u_j,u_j\rangle
\\=\frac1{ T^2}\int_0^T\int_0^T h^n \sum_j f(E_j)
\langle e^{i(s-t) P(h) / h }F_1^* e^{i ( t -s ) P(h) / h } F_1 u_j,u_j\rangle\,dtds.
\end{gathered}
$$
Since $\| e^{i(s-t) P(h) / h }F_1^* e^{i ( t -s ) P(h) / h }F_1
\|_{L^2\to L^2}$ is bounded uniformly in $t,s,h$, we estimate the
integral over the region $|t-s|\leq t_0$ using the upper bound on the
number of eigenvalues, ~\eqref{e:weyl-law},
\begin{equation}
  \label{e:decor-3}
\frac{ 1 } {T^2}\int_{0\leq t,s\leq T\atop |t-s|\leq t_0}h^n 
\sum_j f(E_j)\langle e^{i(s-t) P(h) / h
}F_1^* e^{i ( t -s ) P(h) / h }F_1   u_j,u_j\rangle\,dtds
\leq CT^{-1},
\end{equation}
where $C$ is again a constant independent of $T$ and $h$.

It remains to estimate the integral over the region $t_0\leq |t-s|\leq
T$.  For $t_0\leq |r|\leq T$, define the operator $G_r=e^{ i r P(h) /
h } F_1^* e^{ - i r P(h) / h } F_1$, and rewrite the studied integral
as
$$
\frac{1}{T^2}\int_{0\leq t,s\leq T\atop |t-s|\geq t_0}h^n
\sum_j f(E_j)\langle G_{s-t}u_j,u_j\rangle\,dtds.
$$
Take $\chi\in C_c^\infty(M^\circ)$ such that $F_1=F_1\chi$, and thus
$G_r=G_r\chi$, then by the cyclicity of the trace as in the proof of
Lemma~\ref{l:2}, we can replace $G_{s-t}$ by $\chi G_{s-t}$ with an
$\mathcal O_T(h^\infty)$ penalty. However, by Lemma~\ref{l:hans}, for
$t_0\leq |r|\leq T$,
\begin{equation}
\label{eq:rhorho}
\WFh'(\chi G_r)\subset\{( \rho' , \rho) :
\rho\in K_1\setminus \widetilde K_T,\
\varphi_r(\rho')\in K_1,\
\kappa(\varphi_r(\rho'))=\varphi_r(\kappa(\rho))\}.
\end{equation}
The definition of $ \widetilde K_T $ -- see \eqref{eq:exc} and
\eqref{eq:defKT}~-- shows that the set in \eqref{eq:rhorho}
does not intersect $ \Diag ( T^*M )$. 
This means that the operator $\chi G_r$
satisfies
the hypothesis of Lemma \ref{l:2}, 
and  by~\eqref{e:no-trace} we find
$$
\frac{1}{ T^2}\iint_{\substack{{0\leq t,s\leq T}\\{ |t-s|\geq t_0}}}
h^n\sum_j f(E_j)\langle \chi G_{s-t}u_j,u_j\rangle\,dtds
=\mathcal O_T(h^\infty).
$$
Combining this with~\eqref{e:decor-3} and recalling~\eqref{e:decor-1}
and~\eqref{e:decor-2}, we get
\[
h^n\sum_{ E_j \in [ a , b ] } |\langle F u_j,u_j\rangle|\leq (CT^{-1}+\mathcal O_T(h))^{\frac12},
\]
where $C$ is a constant independent of $T$ and $h$.  By choosing $T$
large and then $h$ small, we obtain~\eqref{eq:prop}.
\end{proof}
%
%%%%%%%%%%%%%%%%%%%%%%%%%%%%%%% END PROP %%%%%%%%%%%%%%%%%%%%%%%%%%%%%%%
%

%%%%%%%%%%%%%%%%%%%%%%%%%%%%%%%%%%%%%%%%%%%%%%%%%%%%%%%%%%%%%%%%%%%%%%%%%%%%%%%%
%%%%%%%%%%%%%%%%%%%%%%%%%%%%%%%%%%%%%%%%%%%%%%%%%%%%%%%%%%%%%%%%%%%%%%%%%%%%%%%%
\section{Quantum ergodicity for restrictions}
\label{qer}

We will now prove Theorem~\ref{t-Sch} and we use the notation from the
second (semiclassical) part of Section~\ref{int}. To simplify the
presentation we put $ Q = id $.  The general case is similar.

We start with some geometric observations.  The condition
\eqref{eq:VE} shows that, in the notation of \eqref{eq:piE}, $ B_E :=
\pi_E ( \Sigma_E ) \subset T^* N , $ is a smooth manifold with a
smooth boundary.  Any $ \rho \in B_E \setminus \partial B_E $ is a
regular value of $ \pi_E $; moreover, $ \pi_E^{-1} ( \rho ) = \{
\rho_+, \rho_- \} $, $ \pi_E $ is a local diffeomorphism
near $ \rho_\pm $, and the involution $ \gamma_E $ is given 
by $ \gamma_E ( \rho_\pm ) = \rho_\mp $.  The Hamilton vector field $ H_p $ is 
transversal to $ \Sigma_ E $ at $ \rho_\pm $.

To prove Theorem~\ref{t-Sch} we can assume that $ [ a, b ] $ is a
small neighbourhood of a fixed energy level $ E $. We then decompose
any compactly supported $ A \in \Psi_h^0 ( N ) $ as follows:
\begin{equation}
\label{eq:AA}   A = \sum_{j=1}^J \widetilde A_{j,\epsilon}  + A_\epsilon + ( 1 -
X_E  )  A , \end{equation}
where
\begin{itemize}
\item  $ X_E \in \Psi_h^{\comp} ( N ) $ is microlocally equal to $ id $ near 
$ B_E \subset T^* N  $, and $ \WFh ( X_E ) $ is contained in small
neighbourhood of $ B_E $,
\item $ \widetilde A_{ j, \epsilon } \in \Psi_h^{\comp} ( N ) $, and  $
\WF_h ( \widetilde A_{j, \epsilon} ) $ is a small open subset of $B_E\setminus \partial B_E$,
\item $ A_\epsilon \in \Psi_h^{\comp} ( N ) $, $ \mu_\sigma ( \WFh (
  A_\epsilon ) ) < \epsilon $, where $\mu_\sigma$ is the symplectic measure.
\end{itemize} 

The estimate~\eqref{e:hs-estimate-2} in the second part of
Lemma~\ref{l:1} shows that the contribution of $ ( 1 - X_E) A $ is
negligible, and that the contribution of $ A_\epsilon $ will disappear
in $ \epsilon \to 0 $ limit.

Hence we only need to prove Theorem~\ref{t-Sch} for terms of the form
$ \widetilde A_{j , \epsilon } $. We assume now that $ \rho \in
B_E\setminus \partial B_E$, and that $\widetilde A\in\Psi_h^{\comp}(N)$
is microlocalized in a small neighborhood $V \subset T^* N $ of $ \rho
$. Choose small $\delta>0$ and define the set
$$
U:=\{\varphi_t(\tilde x,\tilde\xi) \; : \;  |t|<\delta,\
(\tilde x,\tilde\xi)\in \Sigma_{E+\tau}\cap \pi_{E+\tau}^{-1}(V),\ 
|\tau|<\delta\}.
$$
If $V$ and $\delta$ are small enough, then we can write
$
U=U_1\sqcup U_2,
$
where $U_\ell$, $\ell=1,2$, are open subsets of $T^*M$ (one of which
is a neighborhood of $\rho_+$ and the other of $\rho_-$) and moreover,
the maps $\kappa_\ell:U_\ell\to V\times\{|t|,|\tau|<\delta\}$,
\begin{equation}
\label{eq:prkap}
\kappa_\ell \; : \; \varphi_t(\tilde x,\tilde\xi)\longmapsto (\pi_{E+\tau}(\tilde
x,\tilde\xi),t,\tau)\, ,\ \
(\tilde x,\tilde\xi)\in \Sigma_{E+\tau}\cap U_\ell,\ 
|t|,|\tau|<\delta,
\end{equation}
are diffeomorphisms. The maps $\kappa_\ell$ are symplectomorphisms if
we consider $\{|t|,|\tau|<\delta\}$ as a subset of $T^* \mathbb
R_{t}$, with $\tau$ the momentum corresponding to $t$; in fact, we
provide a generating function for $\kappa_\ell$ in~\eqref{eq:psip} below.

Fix a local coordinate system $x=(x',x_n)$ on $M$ such that
$N=\{x_n=0\}$. We identify every half-density $u(x)|dx|^{1/2}$ on $M$
with the function $u(x)$ and every half-density $v(x)|dx'|^{1/2}$ on
$N$ with the function $v(x)$. Consider the operator
\begin{equation}
\label{e:r-def}
\mathcal R:C^\infty(M)\to C^\infty(N\times \mathbb R_t),\ \ 
\mathcal Ru(t) :=(e^{it(P(h)-E)/h}u)|_N,
\end{equation}
then
\begin{equation}
  \label{e:r-intertwining}
hD_t \mathcal R=\mathcal R (P(h)-E).
\end{equation}
Take $X_\ell\in\Psi_h^{\comp}(M)$ microlocalized inside $U_\ell$,
but such that
$$
X_\ell=1\text{ microlocally near }\kappa_\ell^{-1}(\WFh(\widetilde A)\times \{|\tau|,|t|\leq\delta/2\}).
$$
Let $\tilde\chi(t)\in C_{\mathrm c}^\infty(-\delta,\delta)$ be equal to 1 near $[-\delta/2,\delta/2]$. Then
$$
B_\ell:=\tilde\chi(t)\mathcal RX_\ell:C^\infty(M)\to C^\infty(N\times \mathbb R_t)
$$
are compactly microlocalized Fourier integral operators associated to
$\kappa_\ell$. This follows from an oscillatory representation of $
e^{ it ( P(h) - E ) / h }$ given in \cite[\S 10.2]{e-z}. Indeed, in
coordinates $x=(x',x_n)$ and in the notation of \cite[Theorem~10.4]{e-z},
\begin{equation}
\label{eq:Rfio}
{\mathcal B}_{\ell} u ( t , \tilde x ) := 
 {\frac{1}{(2\pi h)^n}}\int_{\RR^n}    \int_{\RR^n}  e^{{\frac             
    i h ( \psi
(t , \tilde x , 0, \eta )- y\cdot \eta )} } b_\ell(t , \tilde x, 0 ,\eta ;h)u(y)
\, dyd\eta  ,
\end{equation}
where
\begin{gather}
\label{eq:psip}
\begin{gathered}
\psi ( 0, x , \eta ) = x\cdot\eta,\
\partial_t \psi ( t , x , \partial_x \psi ) = p ( x , \partial_x\psi
)-E,\\
\varphi_t(x, \partial_x\psi(t,x,\eta))=(\partial_\eta\psi(t,x,\eta),\eta).
\end{gathered}
\end{gather}
The microlocalization inside $ U_\ell $ means that $ \partial_{\xi_n }
p( \tilde x, 0, \xi ) \neq 0 $, and that implies that
$\partial_{(t,\tilde x),\eta}^2\psi $ is nondegenerate. Hence, $ \psi (
t, \tilde x, 0, \eta ) $ is a generating function of $ \kappa_\ell $.

Now, let $u$ be an eigenfunction of $P(h)$ with eigenvalue
$E'=E+\lambda$, where $\lambda\in [-\delta/2,\delta/2]$.  Then
$\WFh(u)\subset p^{-1}([E-\delta/2,E+\delta/2])$ and thus
$$
u|_N=(X_1+X_2)u|_N=B_1u|_{t=0}+B_2u|_{t=0}\ \text{ microlocally near }\WFh(\widetilde A).
$$
Now, by~\eqref{e:r-intertwining} each $w_\ell:=B_\ell u$ solves $
hD_tw_\ell=\lambda w_\ell$ { microlocally near } $ \WFh(\widetilde A)$
for $ |t|\leq\delta/2$.  Therefore, $ w_\ell(t)=e^{it\lambda/h}
w_\ell(0) $ { microlocally near } $\WFh(\widetilde A)$ for $
|t|\leq\delta/2$.

Take $\chi(t)\in C_{\mathrm c}^\infty(-\delta/2,\delta/2)$ that integrates to
$1$. Then
$$
\langle\widetilde A (w_\ell|_{t=0}),w_k|_{t=0}\rangle_{L^2(N)}
=\langle (\chi(t) \otimes \widetilde A)w_\ell,w_k\rangle_{L^2(N\times \mathbb R_t)}
+\mathcal O(h^\infty).
$$
Therefore,
\begin{equation}
\label{eq:Ati}
\begin{split}
\langle\widetilde A(u|_N),(u|_N)\rangle_{L^2(N)} 
& =\sum_{\ell,k =1}^2  \langle\widetilde A (w_\ell|_{t=0}) , w_k |_{t=0}\rangle_{L^2(N)} 
+\mathcal O(h^\infty)\\
& =\sum_{\ell, k=1}^2\langle (\chi(t) \otimes \widetilde A)
w_\ell,w_k\rangle_{L^2(N\times \mathbb R_t)}+\mathcal O(h^\infty) \\
& =\sum_{\ell, k=1}^2\langle B_k^* (\chi(t) \otimes \widetilde A)
B_\ell u,u\rangle_{L^2(M)}+\mathcal O(h^\infty). 
\end{split}
\end{equation}
We now need to analyse the operators
$
B_{k\ell }:=B_{k}^*(\chi(t)\otimes \widetilde A)B_\ell .
$
This is split into two cases. For $k=\ell$, $B_{k \ell}$ is a
pseudodifferential operator and Theorem \ref{t-qe} can be applied with
$ B = B_{\ell \ell}$. Note that the operator $\chi(t)\otimes\widetilde A$
is not pseudodifferential,
in fact its non-semiclassical wavefront set contains points
$(t,t',x,x')$ with $t=t'$ and $(x,x')$ in the support
of the Schwartz kernel of $\widetilde A$. However, $B_{\ell\ell}$ is pseudodifferential
since $B_\ell,B_\ell^*$ are compactly microlocalized and thus we can replace
$\chi(t)$ by a compactly microlocalized operator in $\chi(t)\otimes A$, making the latter
pseudodifferential.

We need to compute the symbol of $B_{\ell \ell}$. For that, we use the
integral representation \eqref{eq:Rfio} and the stationary phase
method, applicable since $ \partial_{(t,\tilde x),\eta}^2 \psi $ is
nondegenerate.  More precisely, the Schwartz kernel, $ B_{\ell}^*
B_\ell ( z, y ) $ is given by
\[
\frac{ 1 } { ( 2 \pi h)^{2n } } \int\limits_{\RR^n_{t,\tilde x} \times
  \RR^n_\eta\times \RR^n_\zeta } e^{ \frac{i} h (  \psi ( t, \tilde x, 0, \eta ) - \psi  ( t
   , \tilde x, 0, \zeta) +  z \cdot \zeta  - y \cdot \eta
   )} b_\ell ( t , \tilde x, 0, \eta ) \overline{ b_\ell ( t , \tilde x, 0, \zeta )}\, d\tilde x dt
d\eta d\zeta .
\]
We apply the method of stationary phase in the $\tilde x,t,\eta$
variables.  The stationary point is given by $ \eta = \zeta $, $
\varphi_t (\tilde x, 0,\partial_x\psi_t(t,\tilde x, 0,\eta) ) = ( y ,
\eta) $, and the value of the phase at the stationary point is $ (z -
y)\cdot \zeta $.  The signature of the Hessian is $ 0 $ and, by
\eqref{eq:psip} and \cite[Theorem 10.6]{e-z}, the leading part of the
symbol in the region of interest is given by
\[ 
\bigg|\frac{  \det \partial_x \partial_\zeta \psi ( t , \tilde x, 0 , \zeta ) }
{  \det \partial_{(t, \tilde x)} \partial_\zeta  \psi( t , \tilde x, 0, \zeta )  
 }\bigg|  
= \frac{1 }{  | \partial_{\xi_n}  p
(  \tilde x , 0 , \partial_x 
\psi ( t , \tilde x, 0, \zeta )) | } , 
\]
where $ t = t ( y, \zeta)$ , $ \tilde x = \tilde x ( y , \zeta ) $,
are the critical points.

Recalling \eqref{eq:prkap}, it follows that
$$
\sigma(B_\ell^*B_\ell)\circ\kappa_\ell^{-1}=|  \partial_{\xi_n } p \circ\kappa_\ell^{-1}\circ \pi_0|^{-1}\text{ near }
\WFh(\widetilde A)\times \{|\tau|,|t|\leq\delta/2\},
$$
where $\pi_0:T^*(N\times \mathbb R_t)\to T^*(N\times \mathbb R_t)$
maps $(\tilde x,\tilde\xi,t,\tau)$ to $(\tilde x, \tilde\xi,0,\tau)$.
From here and by Egorov's Theorem applied to $\chi(t)
\otimes\widetilde A$, we get
$\sigma(B_{\ell\ell})\circ\kappa_\ell^{-1}=| \partial_{\xi_n } p
\circ\kappa_\ell^{-1}\circ\pi_0|^{-1}\sigma(\widetilde A)\chi(t)$ near
$\{|\tau|\leq\delta/2\}$.  Then
\begin{gather}
\label{eq:Ati1}
\begin{gathered}
\negint_{\{p=E_j\}}\sigma(B_{\ell\ell})\,d\mu_{E_j}={\frac 1
  {\mu_{E_j}  ( p^{-1} (E_j ) ) }}
\int_{\{\tau=E_j-E\}}
\sigma(B_{\ell\ell})\circ\kappa_\ell^{-1} \,d \tilde x d \tilde \xi  dt\\
= {\frac 1
  {\mu_{E_j}  ( p^{-1} (E_j ) ) }} \int_{\Sigma_{E_j}\cap U_\ell}
|\partial_{\xi_n} p |^{-1} \pi_{E_j}^* \sigma(\widetilde A) \,d x' d
\xi' ,
\end{gathered}
\end{gather}
where we parametrized $ \Sigma_{E_j } $ by $ ( x', \xi' ) \in B_{E_j} $.

Now, we consider the case $k \neq \ell $. Then $B_{k \ell }$ is a
Fourier integral operator with the canonical transformation $\kappa_{
k \ell } := \kappa_{k}^{-1}\circ\kappa_\ell $.  We want to apply the
decorrelation result given in Lemma \ref{l:fio-decorrelation}.

Using the definition~\eqref{eq:prkap} of $\kappa_\ell$, we see that
the canonical transformation $\kappa = \kappa_{ k \ell } $ can be
described as follows:
$$
\kappa(\varphi_s(\tilde x,\tilde\xi))=\varphi_s(\gamma_{E'}(\tilde x,\tilde\xi)),\
|s|<\delta,\ 
(\tilde x,\tilde\xi)\in\Sigma_{E'}\cap U_\ell.
$$
To apply Lemma \ref{l:fio-decorrelation} we need to verify the following:
there exists $t_0>0$ such that the set
\begin{equation}
\label{eq:EE}
\mathcal E:=\{ \rho \in U_\ell\cap \varphi_{-t}(U_\ell) \; : \; \exists \; t,\ |t|\geq t_0 \,, \ 
\varphi_t (\kappa(\rho )) =\kappa( \varphi_t ( \rho )) \}\subset
T^*M\,, 
\end{equation}
has $\mu_\sigma$-measure zero. To see this, suppose that $\rho\in \mathcal E$, $t$
is the corresponding time, and $s,s'\in (-\delta,\delta)$ are such that
$
\rho=\varphi_s(\tilde x,\tilde \xi)$,
$\varphi_t(\rho)=\varphi_{s'}(\tilde x',\tilde\xi')$,
$(\tilde x,\tilde\xi),(\tilde x',\tilde\xi')\in \Sigma_{E'}\cap U_\ell$.
Then $(\tilde x',\tilde\xi')=\varphi_{t+s-s'}(\tilde x,\tilde \xi)$ and
the condition~$\varphi_t(\kappa(\rho))=\kappa(\varphi_t(\rho))$ can be rewritten as
$$
(\tilde x,\tilde\xi)\in\Sigma_{E'},\
\varphi_{t+s-s'}(\tilde x,\tilde\xi)\in\Sigma_{E'},\
\varphi_{t+s-s'}(\gamma_{E'}(\tilde x,\tilde\xi))
=\gamma_{E'}(\varphi_{t+s-s'}(\tilde x,\tilde\xi)).
$$
Put $t_0>2\delta$, then $t+s-s'\neq 0$.  It now follows
from~\eqref{eq:h-dynl} that the set $\mathcal E$ from~\eqref{eq:EE}
has measure zero; by Lemma~\ref{l:fio-decorrelation}, the
contributions of $ B_{k\ell } $, $ k \neq
\ell $ to the sum~\eqref{eq:qe1}  go to $ 0 $ as $h\to 0$.

Going back to \eqref{eq:Ati}, \eqref{eq:154} and \eqref{eq:Ati1} this means for $
\widetilde A $ satisfying our localization assumptions
\[ 
h^n \sum_{E_j\in [ a, b ] }
\bigg| \langle\widetilde A(u_j|_N),(u_j|_N)\rangle_{L^2(N)} 
- {\frac 1 { V_j } } \sum_{ \ell = 1, 2 } \int_{\Sigma_{E_j} \cap U_\ell
  } |\partial_{\xi_n} p |^{-1} \pi_{E_j}^* \sigma(\widetilde A) \,dx'd\xi'  \bigg| =  o ( 1 ) , 
\]
where $ V_j := {\mu_{E_j} } ( p^{-1} (E_j ) ) $. Assume that the
function $f$ used to define restrictions of half-densities
in~\eqref{eq:resth} is equal to $x_n$, so that $u(x)|dx|^{1/2}$
restricts to $u(\tilde x,0)|dx'|^{1/2}$.  Using the canonical
transformation $\kappa_\ell$, we get the symplectic coordinates
$(x',t,\xi',\tau)$ on $U_\ell$, in which the measure $d\nu_{E_j}$
from~\eqref{eq:mL} is equal to
$V_j^{-1}|\partial_{\xi_n}p|^{-1}\,dx'd\xi'$.  We have thus
proved~\eqref{eq:qe1} with $ f = x_n$; the case of general $f$ follows
by taking the operator
$|\partial_{x_n}f|^{-1/2}A|\partial_{x_n}f|^{-1/2}$ in place of $A$.
This completes the proof of Theorem~\ref{t-Sch}.

%%%%%%%%%%%%%%%%%%%%%%%%%%%%%%%%%%%%%%%%%%%%%%%%%%%%%%%%%%%%%%%%%%%%%%%%%%%%%%%%
%%%%%%%%%%%%%%%%%%%%%%%%%%%%%%%%%%%%%%%%%%%%%%%%%%%%%%%%%%%%%%%%%%%%%%%%%%%%%%%%
\appendix

%%%%%%%%%%%%%%%%%%%%%%%%%%%%%%%%%%%%%%%%%%%%%%%%%%%%%%%%%%%%%%%%%%%%%%%%%%%%%%%%
%%%%%%%%%%%%%%%%%%%%%%%%%%%%%%%%%%%%%%%%%%%%%%%%%%%%%%%%%%%%%%%%%%%%%%%%%%%%%%%%
\section{Semiclassical quantum ergodicity with boundaries}
\label{s:appendix-a}

Let $ ( M , g ) $ be a smooth Riemannian manifold with a piecewise
smooth boundary, $ \partial M $. That means that $ M \subset
\widetilde M $ where $ \widetilde M $ is manifold without boundary to
which $ g $ extends smoothly, and $ \partial M = \bigcup_{j=1}^J N_j $
where $ N_j $ are smooth embedded hypersurfaces in $ \widetilde
M$. Denote by $\partial^\circ M\subset \partial M$ the open set of all
points at which the boundary is smooth, namely points contained in
exactly one of the hypersurfaces $N_j$; the complement $\partial
M\setminus \partial^\circ M$ has measure zero (with respect to the
surface measure on $\partial M$).

We consider an operator $ P ( h ) $ given by 
\eqref{eq:Ph} with Dirichlet boundary conditions. (One can take instead
any self-adjoint boundary conditions, as long as the Weyl
law~\eqref{e:weyl-law} is known to hold.)  Let $p(x,\xi)$ be its
principal symbol; we can extend it smoothly to $T^*\widetilde M$.  We
make the following assumption similar to~\eqref{eq:VE}, for $E\in
[a,b]$:
\begin{equation}
  \label{eq:VE2}
x\in \partial^\circ M,\
V ( x ) = E \implies  d V ( x ) \notin N^*_x \partial M. 
\end{equation}
Then $p^{-1}(E)$ and $T^*_{\partial^\circ M}M$ intersect transversally.
We write
$$
p^{-1}(E)\cap T^*_{\partial^\circ M}M=\Omega_E^+\sqcup \Omega_E^-\sqcup\Omega_E^0,
$$
where $(x,\xi)$ lies in $\Omega_E^+$ if the vector $H_p x\in T
\widetilde M$ is pointing outside of $M$, in $\Omega_E^-$ if this
vector is pointing inside $M$, and in $\Omega_E^0$ if $H_p x$ is
tangent to the boundary of $M$. The covectors in $\Omega_E^0$ are
called glancing, and under the assumption~\eqref{eq:VE2} this set has
measure zero inside $p^{-1}(E)\cap T^*_{\partial^\circ M}M$.

For $(x,\xi)\in p^{-1}(E)$, we define its broken Hamiltonian flow line
$\varphi_t(x,\xi)$ as follows. Assuming without loss of generality
that $t>0$, we consider the Hamiltonian flow line $\exp(tH_p)(x,\xi)$,
defined smoothly on $T^*\widetilde M$, and let $t_0$ be the first
nonnegative time when $\exp(tH_p)(x,\xi)$ hits the boundary. If this
happens at a non-smooth point of the boundary (i.e. on~$\partial
M\setminus \partial^\circ M$), or if $\exp(t_0 H_p)(x,\xi)\in
\Omega_E^0$, then the flow cannot be extended past $t=t_0$. Otherwise,
$\exp(t_0 H_p)(x,\xi)\in\Omega_E^+$ and there exists unique
$(x_0,\xi_0)\in \Omega_E^-$ such that the natural projections of
$\exp(t_0 H_p)(x,\xi)$ and $(x_0,\xi_0)$ onto $T^*\partial M$ are the
same.  We then define $\varphi_t$ inductively, by putting
$\varphi_t(x,\xi)=\exp(tH_p)(x,\xi)$ for $0<t<t_0$ and
$\varphi_t(x,\xi)=\varphi_{t-t_0}(x_0,\xi_0)$ for $t>t_0$.  For any
$T>0$, denote by
\begin{equation}
  \label{e:B-T}
\mathcal B_T\subset T^*M\cap p^{-1}([a,b])
\end{equation}
the closed set of all $(x,\xi)$ such one cannot define the flow
$\varphi_t(x,\xi)$ on the interval $[-T,T]$ using the above
procedure. As shown in~\cite[Lemma~1]{z-z}, for any $T$ the set
$\mathcal B_T\cap p^{-1}(E)$ has measure zero in $p^{-1}(E)$, and for
$|t|\leq T$, $\varphi_t$ is a volume preserving flow on
$p^{-1}(E)\setminus \mathcal B_T$.  See also~\cite[p.~310--311]{H3}
for a symplectically invariant description of the broken Hamiltonian
flow. Since the flow $\varphi_t$ is well-defined almost everywhere,
the standard ergodic theory applies to it.

We will use the following parametrix construction for the
Schr\"odinger propagator away from the set $\mathcal B_T$.  The
following lemma is a rephrasing of results of Christianson~\cite[\S
3.3]{c}:
\begin{lem}
  \label{l:hans}
Fix $T>0$.  Assume that $A\in\Psi_h^{\comp}(M^\circ)$ is supported away
from the boundary of $M$ and $\WFh(A)\subset p^{-1}([a,b])\setminus
\mathcal B_T$.  Then for each $\chi\in C_c^\infty(M^\circ)$ and for
each $t\in [-T,T]$, the operator $\chi e^{-itP/h}A$ is a
Fourier integral operator supported away from $\partial M$ and
associated to the restriction of $\varphi_t$ to a neighborhood of $\WFh(A)\cap
\varphi_t^{-1}(\supp\chi)$, plus an $\mathcal O(h^\infty)_{L^2(M)\to L^2(M)}$
remainder. The following version of Egorov's Theorem holds:
$$
\chi e^{itP/h}A e^{-itP/h}=A_{t,\chi}+\mathcal O(h^\infty)_{L^2(M)\to L^2(M)},
$$
where $A_{t,\chi}\in\Psi_h^{\comp}(M^\circ)$ is supported away from $\partial M$
and $\sigma(A_{t,\chi})=\chi (a\circ \varphi_t)$.
\end{lem}

The following basic Weyl law can be proved for the Dirichlet
realization of $ P ( h ) $ using the standard Dirichlet--Neumann
bracketing method (see~\cite[Chapter 15]{r-s}):
\begin{equation}
  \label{e:weyl-law}
(2\pi h)^n\#\{j: E_j\in [a,b]\}=\mu_\sigma(T^*M\cap p^{-1}([a,b]))+o(1)\text{ as }h\to 0.
\end{equation}
It follows that eigenfunctions cannot on average concentrate near the boundary:
\begin{lem}
  \label{l:nonconcentration}
Assume that $\chi\in C_c^\infty(M^\circ)$ satisfies $0\leq\chi\leq 1$. Then
for $a'<a<b<b'$,
$$
(2\pi h)^n\sum_{E_j\in [a,b]} \int_M (1-\chi)|u_j|^2\,d\Vol\leq \int_{T^*M\cap p^{-1}([a',b'])}
1-\chi\,d\mu_\sigma+o(1)\text{ as }h\to 0.
$$
\end{lem}
\begin{proof}
Take $f\in C_c^\infty(a',b')$ such that $0\leq f\leq 1$ and $f=1$ near $[a,b]$. Since
$f$ and $1-\chi$ are nonnegative, it suffices to show that
$$
(2\pi h)^n\sum_j \int_M (1-\chi)f(E_j)|u_j|^2\,d\Vol=\int_{T^*M}(1-\chi)f(p)\,d\mu_\sigma+o(1)
\text{ as }h\to 0.
$$
This holds since the asymptotics for $1$ in place of $1-\chi$ follows from~\eqref{e:weyl-law},
while the asymptotics for $\chi$ follows from~\eqref{e:trace-xp}.
\end{proof}

We can now prove quantum ergodicity for manifolds
with boundary:
\begin{theo}
\label{t-qe}
Suppose that $ (M , g ) $ is a compact manifold with a
piecewise smooth boundary and 
 that $ u_j = u_j ( h) $ are 
normalized eigenfuctions of the Dirichlet 
realization of $ P (h)$. If~\eqref{eq:ener} and~\eqref{eq:VE2} hold, then
for any $ B \in \Psi_h^0 ( M^\circ ) $ compactly supported away from $ \partial M $, 
\begin{equation}
\label{eq:154}
h^n\sum_{E_j\in [a,b] }\bigg|\langle B
u_j,u_j\rangle_{L^2 ( M ) } 
- \negint_{p^{-1}(E_j)}\sigma(B)\,d\mu_{E_j}\bigg|\longrightarrow 0\,, \ \ h \to 0 .
\end{equation}

\end{theo}
\begin{proof}
Take $a',b'$ such that $a'<a<b<b'$ and~\eqref{eq:ener}
and~\eqref{eq:VE2} hold for $E\in [a',b']$. (If the flow is no longer
ergodic on $p^{-1}(E)$ when $E\not\in[a,b]$, we would need to consider
$a',b'$ close to $a,b$, for example $a'=a-1/T$ and $b'=b+1/T$ and
crudely estimate the contribution of $[a',b']\setminus [a,b]$ by the
Weyl law.) Take large $T>0$ and choose a cutoff function $\chi_T\in
C_c^\infty(M^\circ)$ such that $0\leq\chi_T\leq 1$ and
$$
\int_{T^*M\cap p^{-1}([a'-1,b'+1])}1-\chi_T\,d\mu_\sigma\leq T^{-1}.
$$
Let the function $\psi\in C_c^\infty(a'-1,b'+1)$ satisfy
$$
\psi(E)\int_{p^{-1}(E)}\chi_T\,d\mu_E=\int_{p^{-1}(E)}\sigma(B)\,d\mu_E,\ E\in [a',b'].
$$
By Lemma~\ref{l:nonconcentration}, it is enough to show that for
$T$ arbitrarily large but fixed, \eqref{eq:154} holds for the operator
$B-\psi(P(h))\chi_T$, whose symbol integrates to zero on $p^{-1}(E)$ for
$E\in [a',b']$; therefore, without loss of generality we assume that
\begin{equation}
  \label{eq:zeroint}
\int_{p^{-1}(E)}\sigma(B)\,d\mu_E=0,\
E\in [a',b'].
\end{equation}
By the elliptic estimate (see for instance~\cite[Proposition~3.2]{d-g}),
we may assume that $\WFh(B)\subset
p^{-1}((a',b'))$ and in particular $B\in\Psi_h^{\comp}$.  The set
$\mathcal B_T$ defined in~\eqref{e:B-T} is closed and has measure
zero; therefore, we can write $B=B'_T+B''_T$, where $\WFh(B'_T)\cap
\mathcal B_T=\emptyset$ and
$\|\sigma(B''_T)\|_{L^2(p^{-1}[a',b'])}\leq
T^{-1}$. By~\eqref{e:hs-estimate}, the contribution of $B''_T$
to~\eqref{eq:154} goes to zero in the limit $\lim_{T\to
\infty}\limsup_{h\to 0}$; therefore, we can replace $B$ by $B'_T$
in~\eqref{eq:154}.  Define the quantum averaged operator
$$
\langle B'_T\rangle_T:={1\over T} \int_0^T e^{itP/h}B'_Te^{-itP/h}\,dt.
$$
Then by Lemma~\ref{l:hans}, $\langle B'_T\rangle_T\chi_T$ is, up to an
$\mathcal O(h^\infty)_{L^2\to L^2}$ remainder, a pseudodifferential
operator in $\Psi_h^{\comp}$ compactly supported inside $M^\circ$ and
with principal symbol
$$
\sigma(\langle B'_T\rangle_T\chi_T)=\chi_T\langle\sigma(B'_T)\rangle_T={\chi_T\over T}
\int_0^T \sigma(B'_T)\circ \varphi_{t}\,dt.
$$
Since each $u_j$ is an eigenvalue of $P(h)$, we can write
the left-hand side of~\eqref{eq:154} as
$$
h^n\sum_{E_j\in [a,b]}|\langle \langle B'_T\rangle_T u_j,u_j\rangle|;
$$
using the Weyl law~\eqref{e:weyl-law} and Cauchy--Schwarz, we see that it remains to prove that
$$
\lim_{T\to \infty}\limsup_{h\to 0}h^n\sum_{E_j\in [a,b]}\|\langle B'_T\rangle_T u_j\|_{L^2}^2=0.
$$
We can replace $\langle B'_T\rangle_T$ here
by $\langle B'_T\rangle_T\chi_T$, as
$$
\limsup_{h\to 0}(2\pi h)^n\sum_{E_j\in [a,b]}\|(1-\chi_T)u_j\|_{L^2}^2\leq T^{-1}
$$
by Lemma~\ref{l:nonconcentration}.
By~\eqref{e:hs-estimate}, it remains to show that
$$
\lim_{T\to \infty}\|\sigma(\langle
B'_T\rangle_T\chi_T)\|_{L^2(p^{-1}([a',b']))} \leq
\lim_{T\to \infty}\|\langle\sigma(B'_T)\rangle_T  \|_{L^2(p^{-1}([a',b']))} 
 = 0.
$$
For this, we write for each $E\in [a',b']$,
$$
\|\langle\sigma(B'_T)\rangle_T\|_{L^2(p^{-1}(E))}
\leq \|\langle \sigma(B)\rangle_T\|_{L^2(p^{-1}(E))}
+\|\langle \sigma(B''_T)\rangle_T\|_{L^2(p^{-1}(E))}.
$$
The first term on the right-hand side converges to 0 when $T\to
\infty$ by~\eqref{eq:zeroint} and the von Neumann ergodic theorem (see
for example \cite[Theorem 15.1]{e-z}), while the second term is
bounded by $\|\sigma(B''_T)\|_{L^2(p^{-1}(E))}$.
\end{proof}

%%%%%%%%%%%%%%%%%%%%%%%%%%%%%%%%%%%%%%%%%%%%%%%%%%%%%%%%%%%%%%%%%%%%%%%%%%%%%%%%
%%%%%%%%%%%%%%%%%%%%%%%%%%%%%%%%%%%%%%%%%%%%%%%%%%%%%%%%%%%%%%%%%%%%%%%%%%%%%%%%
\section{From semiclassical to high energy asymptotics}
\label{s:appendix-b}

In this appendix we specialize to $ P ( h ) = - h^2 \Delta_g $ and 
show how Theorem \ref{t-Sch} implies the results of
\cite{t-z}. 

Suppose that $ ( M , g ) $ is a compact Riemannian manifold 
with a piecewise smooth boundary in the sense of Appendix~\ref{s:appendix-a}, 
and with an ergodic broken geodesic flow  $ \varphi_t : S^*M \to S^*M $.
Suppose that $ N \subset M$ is 
an open smooth hypersurface whose closure is disjoint from the boundary.
The energy surface $p^{-1}(1)=S^*M$ is the cosphere bundle of $M$ and
$\Sigma_1=S^*_NM$ is the restriction of $S^*M$ to $N$;
$B_1=\pi_1(\Sigma_1)=B^*N$ is the coball bundle of $N$ and
$\gamma_1:\Sigma_1\to \Sigma_1$ is the reflection across the
orthogonal complement of the conormal bundle $N^*N\subset T^*_NM$.

The dynamical assumption~\eqref{eq:h-dynl} becomes
\begin{gather}
\label{eq:h-dyn}
\begin{gathered}
\text{ The set of $ \rho  \in 
  S_N^*M$ satisfying }
\varphi_t ( \rho)\in S^*_NM\\
\text{and }\varphi_t ( \gamma_1 ( \rho)) = \gamma_1 ( \varphi_t (\rho  ) )
\ \text{ for some $ t \neq 0 $, has measure $ 0 $.}
\end{gathered}
\end{gather}

Let $ \{ u_j\}_{j=0}^\infty  $ be the complete set 
of  eigenfunctions of the Laplacian  on $ ( M , g ) 
$:
\[  - \Delta_g u_j = \lambda_j^2 u_j , \ \ 
\| u_j \|_{L^2} =1 , \ \  0 = \lambda_0 < \lambda_1 \leq \lambda_2
\leq \cdots . \]

The statement of the theorem uses the standard concept of a
(nonsemiclassical) pseudodifferential operator on a manifold~-- see
\cite[\S 18.2]{H3}.

\begin{theo}
\label{t-Lap}
Let $ N $ be a smooth open hypersurface satisfying \eqref{eq:h-dyn}
with closure disjoint from the boundary. 
Suppose that $ A \in \Psi^0_{\rm phg} ( N ) $ is a  classical 
pseudodifferential operator on $ N $, compactly supported inside $ N $.
Put $    v_j :=  u_j |_ N  $. Then
\begin{equation}
\label{eq:qe}  \frac{1} { \lambda^n} \sum_{ \lambda_j \leq \lambda }
\bigg|  \langle A v_j , v_j \rangle_{L^2 ( N, d \vol_{g|_N} )} -
  {\int \! \! }_{ S^*_N M }  \pi_1^* \sigma ( A ) 
\,  d \nu_1   \bigg| \longrightarrow 0 ,  \
\ \ \lambda \to \infty ,  \end{equation}
where $ \sigma ( A ) $  is the principal symbol of $ A $ (a homogeneous
function of degree $ 0 $ on $ T^*N \setminus \{ 0 \} $), and  the measure
$ \nu_1 $ is defined in \eqref{eq:mN}.
\end{theo}

\noindent{\bf Remark.} 
Theorem \ref{t-Sch} allows more general restrictions $ a u_j |_N + b
\lambda_j^{-1} \partial_\nu u_j |_N $, for $a,b\in C^\infty(N)$. We
note also that the nonsemiclassical formulation of quantum ergodicity
only implies the angular equidistribution of $ v_j $ in $ T_x^* N
$. That is natural for the standard quantum ergodicity since $ u_j $
concentrate on $ S^* M $ but not in this case as $ v_j $'s can be
microsupported anywhere in $ B^* N $. That is remedied in the
semiclassical Theorem \ref{t-Sch}.

\begin{proof}
To show how Theorem~\ref{t-Lap} follows from Theorem~\ref{t-Sch} we
put $ V \equiv 0 $ and identify $ L^2 ( M , \Omega^{1/2}_M ) $ with $
L^2 ( M , d \vol_g ) $ by writing half-densities as $ u ( x ) | d
\vol_g |^{1/2} $, where $ u \in L^2 ( M , d \vol_g ) $.

Let $ x= ( x', x_n ) $ be normal geodesic coordinates near $ ( 0 , 0 )
\in N$, in which $ N = \{ x_n = 0 \} $, $ p ( x , \xi', \xi_n) =
\xi_n^2 + r ( x, \xi') $, and $ r(x', 0,\xi') $ is the dual of the
restriction metric $g|_N$. Suppose that $ f $ satisfies $ f |_N = 0 $,
$ |d f ( x ) |_{g} = 1 $. In the chosen coordinates the last condition
means that $ \partial_{x_n} f = 1 $, $ f = 0 $, on $ N $.  Hence the
restriction of half-densities \eqref{eq:resth} obtained using this
choice of $ f $ shows that, we obtain an identification with the
restriction of functions $ u | _N \in L^2 ( N , d \vol_{g|_N }) $.

We now write out locally the measure $\nu_1$ from~\eqref{eq:mL}.
In our coordinates, $ S^* M $ can be parametrized by $ ( 
x' , x_n,  \xi' ) $, $ r ( x' , \xi') \leq 1$, $ \xi_n = \pm ( 1 - r ( x,
\xi'))^{1/2} $ (the parametrization degenerates at $ r ( x, \xi')
=1$). The Liouville measure is obtained by requiring $ d \mu_1 \wedge
d p = dx d \xi $, and
\[
d \mu_1 = \frac{ 1 }{ 2 | \xi_n| } dx  d \xi' = 
\frac{ 1 }{ 2 ( 1 - r ( x,\xi'))^{1/2} } dx  d \xi' .
\]
In the notation of \eqref{eq:mN} this gives
\[
d \nu_1  =  \frac{ 1 } { \mu_1 ( S^*M ) }\frac1 { 2 \sqrt{ 1 - r ( x', 0, \xi') }   } dx' d \xi' ,
\ \  \ \  B_1 = \{ ( x' , \xi' ) :  r ( x', 0, \xi' ) \leq 1 \} \,,  
\]
where we parametrized $ S_N^*M = \{ ( x' , 0, \xi ) : \xi_n^2 + r (
x', 0, \xi' ) = 1 \} $ by $ ( x', \xi' ) \in B_1 $.

To pass from the semiclassical result to the special case of the high
energy result we put $ E_j = h^2 \lambda_j ^2 $, $ h = 1/\lambda $.
The difficulty lies in controlling low frequency contributions and
estimates \eqref{e:hs-estimate-2} and \eqref{e:hs-estimate-3} are
crucial for that.

Let $\widehat A $ be a classical pseudodifferential operator of order
$ 0 $ on $ N$, with a compactly supported Schwartz kernel in $ N
$. (Henceforth operators with hats denote polyhomogeneous operators,
while operators without hats denote semiclassical operators.)  Its
principal symbol $ \sigma ( \widehat A ) $ is a homogeneous function
of degree $ 0 $ on $ T^*N $. We define $ A_\epsilon \in \Psi_h^0 (N ) $
by putting
\[
A_\epsilon := \Op_h \big( \sigma ( \widehat A) ( 1 - \chi ( | \xi'|_{ g|_ N }/
\epsilon  ) )\big) , \ \
\chi \in C^\infty_{\rm{c}}  ( {\mathbb R}) \, ,\ \chi ( t  ) = 1 , \ \ |
t |  \leq 1 .
\]
Theorem~\ref{t-Sch} shows that for $ 0 < a < b $, and $ v_j = u_j
|_N$,
\begin{equation}
\label{eq:qe2}  h^n \sum_{  h\lambda_j  \in  [ a, b ] }
\bigg|    \langle A_\epsilon   v_j  , v_j  \rangle_{ L^2 ( N , d \vol_{g|_N} ) \
}   - {\int \! \! }_{ \Sigma_{E_j}  }
\pi_{E_j  } ^ *\sigma ( A_\epsilon )  \,d \nu_{E_j  }
 \bigg| \longrightarrow 0 ,  \ \ h = 1/\lambda \to 0 .
\end{equation}
We also have 
\[
{\int \! \! }_{ \Sigma_{E_j}  } 
\pi_{E_j  } ^ *\sigma ( A_\epsilon  )  d \nu_{E_j  }  = 
 {\int \! \! }_{ S^*_N M   }  \pi_1^* \sigma ( \widehat A ) \,d \nu_1 +
 {\mathcal O} ( \epsilon )  ,
\]
and hence the result will follow once we show that
\begin{equation}
\label{eq:qee} 
 h^n \sum_{ h \lambda_j  \in [ a, b ] } 
 \left|    \langle ( \widehat A - A_\epsilon )  v_j  , v_j  \rangle_{ L^2 ( N , d
     \vol_{g|_N} ) }
\right|  =  {\mathcal O} ( \epsilon ) + {\mathcal
     O}_\epsilon ( h ) .
\end{equation}
Indeed, \eqref{eq:qe2} and~\eqref{eq:qee} together give, for $[a,b]=[1,2]$,
\begin{equation}
  \label{e:qez}
h^n\sum_{h\lambda_j\in [1,2]}
\bigg|\langle \widehat A v_j,v_j\rangle_{L^2(N,d\vol_{g|_N})}-{\int\!\!}_{S^*_NM}\pi_1^*\sigma(\widehat A)\,d\nu_1
\bigg|\to 0\text{ as }h\to 0.
\end{equation}
Summing~\eqref{e:qez} for $h=2^k\lambda^{-1}$, $1\leq k\leq\log_2\lambda$, we get~\eqref{eq:qe}.

We now prove~\eqref{eq:qee}. Using \eqref{e:hs-estimate-3}, it will follow from
\begin{equation}
\label{eq:qeee}
\lim_{\epsilon\to 0}\limsup_{h\to 0}h^n \sum_{ h \lambda_j  \in [ a, b ] } 
\|  (\widehat A-A_\epsilon)   v_j(h) \|^2=0.
\end{equation}
For this, we first claim that for any vector field $X$ on $N$,
\begin{equation}
  \label{e:key-claim}
\|(\widehat A-A_\epsilon)hX\|_{L^2\to L^2}\leq C\epsilon+\mathcal O_\epsilon(h).
\end{equation}
Indeed, the left-hand side of \eqref{e:key-claim} is~$\mathcal O(h)$
if we put an operator in the class $\Psi^{-1}_{\mathrm{phg}}$ or
$h\Psi_h^{-1}$ in place of $\widehat A-A_\epsilon$, which means that we
can reduce to local coordinates, in which we can assume
$X=\partial_{y_1}$ and the full symbol of $(\widehat A-A_\epsilon)hX$
in the non-semiclassical left quantization becomes, up to
$\Psi^{-1}_{\mathrm{phg}}+h\Psi_h^{-1}$ terms,
$$
r(y,\eta;h):=h\eta_1 a^0(y,\eta/|\eta|)\big(
(1-\chi(|\eta|))-(1-\chi(h|\eta|_{g|_N}/\epsilon))\big).
$$
However, $\partial^\alpha_y \partial^\beta_\eta r(y,\eta;h)=\mathcal
O(\epsilon+h)\langle\eta\rangle^{-|\beta|}$ (here the first cutoff
gives the $\mathcal O(h)$ term, while the second cutoff gives the
$\mathcal O(\epsilon)$ term); therefore, by the $L^2$ boundedness of
classical pseudodifferential operators, we get~\eqref{e:key-claim}.

Now, let $B_0\in\Psi_h^{\comp}(N)$ be a semiclassical pseudodifferential
operator equal to the identity microlocally near the zero section of
$T^*N$, but supported inside an $\epsilon^{1/2}$ sized neighborhood of
the zero section. Then we can write
$$
1-B_0=\sum_k (hX_k)B^k_0+\mathcal O_\epsilon(h)_{L^2\to L^2}
$$
for
some vector fields $X_k$ (independent of $\epsilon$) and some
$B^k_0\in\Psi_h^{\comp}(N)$ (with $L^2\to L^2$ norm $\mathcal O(\epsilon^{-1/2})$);
by~\eqref{e:key-claim}, we have
$$
\|(\widehat A-A_\epsilon)(1-B_0)\|_{L^2\to L^2}\leq C\epsilon^{\frac12}+\mathcal
O_\epsilon(h)
$$
and thus by~\eqref{e:hs-estimate-3}, the estimate~\eqref{eq:qeee}
holds for $(\widehat A-A_\epsilon)(1-B_0)$. Same estimate holds for
$(\widehat A-A_\epsilon)B_0$, by recalling that $\|\widehat
A-A_\epsilon\|_{L^2\to L^2}=\mathcal O(1)$ and
using~\eqref{e:hs-estimate-2} together with the bound
$\|\sigma(B_0)\|_{L^2}\leq\epsilon^{(n-1)/4}$.  This finishes the proof
of~\eqref{eq:qeee} and thus of Theorem~\ref{t-Lap}.
\end{proof}

% arXiv bibliography macro
\def\arXiv#1{\href{http://arxiv.org/abs/#1}{arXiv:#1}}


\begin{thebibliography}{0}

\bibitem{b-g-t} N.~Burq, P.~G\'erard and N.~Tzvetkov, 
{\em Restrictions of the Laplace--Beltrami eigenfunctions to
  submanifolds}, Duke Math. J. {\bf 138}(2007), no.~3, 445--486.

\bibitem{c} H.~Christianson, 
{\em Quantum monodromy and non-concentration near a closed
  semi-hyperbolic orbit,}
Trans. Amer. Math. Soc. {\bf 363}(2011), no.~7, 3373--3438.

\bibitem{c-t-z} H.~Christianson, J.~Toth, and S.~Zelditch,
{\em Quantum ergodic restriction for Cauchy data:
interior QUE and restricted QUE,\/}
preprint, \arXiv{1205.0286}.

%\bibitem{c-t-z2} H.~Christianson, J.~Toth, and S.~Zelditch,
%{\em Quantum flux and quantum ergodicity for cross sections,\/}
%in preparation.

\bibitem{d} S.~Dyatlov, 
{\em Asymptotic distribution of quasi-normal modes for Kerr--de Sitter
  black holes,\/}
to appear in Ann. Henri Poincar\'e, \arXiv{1101.1260}.

\bibitem{d-g} S.~Dyatlov and C.~Guillarmou, 
{\em Microlocal limits of plane waves and Eisenstein functions,\/}
preprint, \arXiv{1204.1305}

\bibitem{GeLe} P.~G\'erard and E.~Leichtnam,
{\em Ergodic properties of eigenfunctions for the Dirichlet problem,\/}
Duke Math. J. \textbf{71}(1993), no.~2, 559--607.

\bibitem{HMR} B.~Helffer, A.~Martinez, and D.~Robert,
{\em Ergodicit\'e en limite semi-classique,\/}
Comm. Math. Phys. {\bf 109}(1987), no.~2, 313--326.

\bibitem{H3} L.~H{\"o}rmander,
{\em The Analysis of Linear Partial Differential Operators III. Pseudo-differential Operators,\/}
Springer, 1985.

\bibitem{r-s} M.~Reed and B.~Simon, {\em
Methods of Modern Mathematical Physics IV. Analysis of Operators,\/}
New York, Academic Press, 1978.

\bibitem{s-z} J.~Sj\"ostrand and M.~Zworski, 
{\em Quantum monodromy and semiclassical trace formulae,\/}
J. Math. Pure Appl. {\bf 81}(2002), no.~1, 1--33.

\bibitem{ta} M.~Tacy, 
{\em Semiclassical $L^p$ estimates of quasimodes on submanifolds,\/}
Comm. P.D.E. {\bf 35}(2010), no.~8, 1538--1562. 

\bibitem{t-zb} J.A.~Toth and S.~Zelditch, 
{\em Quantum ergodic restriction theorems. I: interior hypersurfaces
  in domains with ergodic billiards,\/} Ann. Henri Poincar\'e
{\bf 13}(2012), 599--670.

\bibitem{t-z} J.A.~Toth and S.~Zelditch, 
{\em Quantum ergodic restriction theorems, II: manifolds without
  boundary,\/}
preprint, \arXiv{1104.4531}.

\bibitem{z-z} S.~Zelditch and M.~Zworski,
{\em Ergodicity of eigenfunctions for ergodic billiards,\/}
Comm. Math. Phys. \textbf{175}(1996), no.~3, 673--682.

\bibitem{e-z} M.~Zworski,
{\em Semiclassical analysis,\/}
Graduate Studies in Mathematics 138, AMS, 2012.

\end{thebibliography}
\end{document}